\documentclass[10pt]{amsart}

\usepackage{amsfonts}
\usepackage{amssymb}
\usepackage{amsmath, amsthm, latexsym}
\usepackage[all]{xy}
\usepackage{hyperref}
\usepackage[margin=1.5in]{geometry}
\usepackage{color}

\def \C{\mathbb{C}}
\def \Z{\mathbb{Z}}
\def \R{\mathbb{R}}

\def \Q{\mathbb{Q}}

\def \L{\mathcal{L}}

\def \k{{\bf k}}

\def \V{\mathcal{V}}
\def \D{\mathcal{D}}
\def \P{\mathcal{P}}

\def \GL{\textup{GL}}
\def \SL{\textup{SL}}

\def \vol{\textup{vol}}
\def \End{\textup{End}}
\def \Fl{F\ell}
\def \GZ{\textup{GZ}}
\def \w{{\underline{w}_0}}

\theoremstyle{plain}
\newtheorem{Th}{Theorem}[section]
\newtheorem{Lem}[Th]{Lemma}
\newtheorem{Prop}[Th]{Proposition}
\newtheorem{Cor}[Th]{Corollary}

\theoremstyle{definition}
\newtheorem{Ex}[Th]{Example}
\newtheorem{Def}[Th]{Definition}
\newtheorem{Rem}[Th]{Remark}
\newtheorem*{REM}{Remark}

\begin{document}

\title{On a notion of anticanonical class for families of convex polytopes}
\author{Kiumars Kaveh}
\address{Department of Mathematics, University of Pittsburgh,
Pittsburgh, PA, USA.}
\email{kaveh@pitt.edu}
\author{Elise Villella}
\address{Department of Mathematics, University of Pittsburgh,
Pittsburgh, PA, USA.}
\email{emv23@pitt.edu}


\thanks{The first author is partially supported by a
Simons Foundation Collaboration Grants for Mathematicians and a National Science Foundation Grant (DMS-1601303).}

\keywords{Convex polytope, virtual polytope, Gelfand-Zetlin polytope, (anti)canonical class, polytope algebra, group compactification, toric variety, spherical variety} 
\subjclass[2010]{14M25, 52B20, 14M27}
\date{\today}
\maketitle

\setcounter{tocdepth}{1}
\tableofcontents


\begin{center}
{\bf This is a preliminary version, comments are welcome.}
\end{center}


\section*{Introduction}
The purpose of this note is to give a generalization of the statement that the anticanonical class of a (smooth) projective toric variety is the sum of invariant prime divisors, corresponding to the rays in its fan (or facets in its polytope), to some other classes of varieties with algebraic group actions. 

To this end, we suggest an analogue of the notion of anticanonical class (of a compact complex manifold) for linear families of convex polytopes. This is inspired by the Serre duality for smooth projective varieties as well as the Ehrhart-Macdonald reciprocity for rational polytopes. The main examples we have in mind are: (1) The family of polytopes normal to a given fan (which corresponds to the case of toric varieties). (2) The family of Gelfand-Zetlin polytopes (which corresponds to the case of the flag variety). (3) The family of Newton-Okounkov polytopes for a (smooth) group compactification. 
More generally, our approach can be applied to projective spherical varieties as well as Schubert and Bott-Samelson varieties.

We should point out that formulae for anticanonical classes of projective spherical varieties are known (\cite{Brion-curve}, \cite[Proposition 4]{Rittatore} and \cite[Section 3]{GH}). These formulae represent the anticanonical class as a linear combination of the $G$-stable divisors and the $B$-stable divisors that are not $G$-stable (the so-called {\it colors}). Our method identifies the (virtual) polytope associated to the anticanonical class and hence describes the anticanonical line bundle. Nevertheless, we hope the present article provides a combinatorial and convex geometric angle on the notion of canonical class from geometry.

All the varieties we consider are over an algebraically closed field of characteristic $0$. Let $X$ be an $n$-dimensional normal projective variety with an action of an algebraic group $G$. Suppose there exists a full dimensional cone $\mathcal{C}$ in the $G$-ample cone of $X$ such that for any $G$-linearized very ample line bundle $L \in \mathcal{C}$ we can associate a rational convex polytope $\Delta(L) \subset \R^n$ such that the following hold:
\begin{itemize}
\item[(i)] For any two very ample $L, L' \in \mathcal{C}$ we have $\Delta(L \otimes L') = \Delta(L) + \Delta(L')$, where the righthand side is the Minkowski addition of polytopes.
\item[(ii)] For any very ample $L \in \mathcal{C}$ we have:
\begin{equation}  \label{equ-intro-dim-H^0}
\dim(H^0(X, L)) = |\Delta(L) \cap \Z^n|.
\end{equation}
In other words, the Ehrhart function of $\Delta(L)$ coincides with the Hilbert function of the ring of sections $A(L) = \bigoplus_{k \geq 0} H^0(X, L^{\otimes k})$. 
\end{itemize}

The property (i) above implies that $\Delta: L \mapsto \Delta(L)$ can be extended by linearity to a map from $\textup{Pic}_G(X)$ to the vector space of virtual polytopes in $\R^n$ (see Section \ref{subsec-prelim-polytopes} for a review of the notion of virtual polytope).

Toric varieties and flag varieties (of general connected reductive groups) are two classes of examples where we have such assignment of polytopes to line bundles. 
A large class of varieties where one has such assignement of polytopes to line bundles is the class of spherical varieties (see Remark \ref{rem-NO-polytope-spherical}). 
Beside toric varieties and flag varieties, this includes the class of group compactifications (see Section \ref{subsec-antican-group-compactification}). This assignment satisfies the property (ii) above. The property (i) may fail if we consider the whole ample cone (see \cite[Example 3.2]{Kaveh-note-spherical}). Nevertheless, in light of Proposition \ref{prop-family-piecewise-lin}, the property (i) holds for some subcone of maximal dimension in the ample cone.

Yet other classes of varieties where one can assign polytopes to line bundles are Schubert varieties (of general connected reductive groups) as well as their Bott-Samelson varieties (see \cite{Caldero, Fujita, HY, MSS, Valya-Bott-Samelson}). 


When $X$ is smooth, by the Serre duality theorem the anticanonical class $K_X^{-1} \in \textup{Pic}(X)$ has the property that: 
\begin{equation}  \label{equ-intro-linebundle}
\dim(H^0(X, L \otimes K_X^{-1})) = \chi(X, L \otimes K_X^{-1}) = (-1)^n \chi(X, L^{-1}),
\end{equation}
where $\chi$ denotes the Euler characteristic.
This combined with the Ehrhart-Macdonald reciprocity implies that: 
\begin{equation}   \label{equ-intro-antican-linebundle2}
|\Delta(L \otimes K_X^{-1}) \cap \Z^n| = |\Delta^\circ(L) \cap \Z^n|, 
\end{equation}
where $P^\circ$ denotes the relative interior of a polytope $P$.

Let $\P_n(\Q)$ denote the collection of all rational convex polytopes in $\R^n$. It is closed under Minkowski addition and multiplication by positive rational numbers. We also denote by $\V_n(\Q)$ the $\Q$-vector space of virtual polytopes spanned by $\P_n(\Q)$ (see Section \ref{subsec-prelim-polytopes}). Let $C$ be a full dimensional rational convex polyhedral cone in some finite dimensional vector space $V$. We call a linear map $\Delta: C \to \P_n(\Q)$ a {\it linear family of convex polytopes}. 

In Section \ref{subsec-polytope-alg}, extending the Khovanskii-Pukhlikov description of the cohomology ring of a smooth projective variety (\cite{KhP}), we associate an algebra to a linear family of polytopes which we call its {\it polytope algebra}. In \cite{Kaveh-note-spherical} the polytope algebra of the Gelfand-Zetlin family is used to recover the Borel description of the cohomology ring of the flag variety. The interesting work \cite{KST} uses polytope algebra of the Gelfand-Zetlin family to address problems in Schubert calculus. 

We will use \eqref{equ-intro-antican-linebundle2} as the definition of an anticanonical (virtual) polytope for a linear family of polytopes. More precisely, let $\Gamma \subset V$ be a full rank lattice in $V$. We say that a virtual polytope $\Delta(\kappa)$, $\kappa \in \Gamma$, is an anticanonical (virtual) polytope in the family with respect to $\Gamma$ if for all $\gamma \in C^\circ \cap \Gamma$ such that $\gamma - \kappa \in C \cap \Gamma$ we have:
\begin{equation}   \label{equ-intro-antican-polytope}
|\Delta(\gamma - \kappa) \cap \Z^n| = |\Delta^\circ(\gamma) \cap \Z^n|. 
\end{equation}
(Here $C^\circ$ is the interior of the cone $C$.)

The main observations of this note are the following (Theorem \ref{th-main}): (1) an anticanonical polytope for a family $\Delta$ (if it exists) is unique. (2) Extending the case of toric varieties, the anticanonical polytope is represented (in the polytope algebra) by the sum of rays in the normal fan to the family.

In Section \ref{sec-examples} we will use uniqueness of the anticanonical polytope to recover, in a simple purely combinatorial way, formulae for the anticanonical classes of (smooth) toric varieties, complete flag variety as well as (smooth) group compactifications. 

Our approach might be useful in the study of Fano varieties and in fact we suggest a convex geometric notion of a Fano variety. Namely, we call a linear family of polytopes a {\it Fano family} when the anticanonical element is represented by a polytope in the family, as opposed to just a (virtual) polytope in the group generated by the family (see Remark \ref{rem-Fano-var}).

\begin{REM}
In \cite{KST}, to each Schubert variety (in type A) the authors correspond a collection of certain faces of a Gelfand-Zetlin polytope.   
One knows that the anticanonical class of the flag variety is twice the sum of Schubert classes of codimension $1$. Our description of the anticanonical class of the flag variety as the sum of all the facets of a Gelfand-Zetlin polytope agrees with that of \cite{KST}.
\end{REM}

\begin{REM}  
We recall that the notion of canonical class of a smooth projective variety extends to normal varieties. 
In fact if $X$ is normal and Cohen-Macaulay, the equation \eqref{equ-intro-linebundle} still holds (see \cite[Chapter III, Corollary 7.7]{Hartshorne}). In this case, the anticanonical class is in general a Weil divisor. Thus one should be able to use our approach provided that there is a linear family $\Delta: C \to \P_n(\Q)$ where $C$ is a full dimensional cone in the $G$-linearized divisor class group of the variety. This is the case, for example, for spherical varieties and Schubert varieties (in particular spherical varieties and Schubert varieties are Cohen-Macaulay).
\end{REM}


{\bf Ackknowledgement:} We would like to thank Valentina Kiritchenko for useful discussions and suggestions. In particular, she suggested looking at an analogue of the notion of a Fano variety in the context of linear families of polytopes. 


\section{Linear families of convex polytopes and anticanonical polytope}   \label{sec-lin-family-polytope}
\subsection{Preliminaries about polytopes}   \label{subsec-prelim-polytopes}
Let $\P_n$ denote the collection of all convex polytopes in $\R^n$. The set $\P_n$ is closed under Minkowski addition and multiplication by positive scalars.  One knows that the Minkowski addition in $\P_n$ is cancellative, i.e. if for $P_1, P_2, P \in \P_n$ we have $P_1 + P = P_2 + P$ then $P_1 = P_2$. It follows that $\P_n$ can be formally extended to a vector space $\V_n$. The elements of $\V_n$ are formal differences $P_1 - P_2$ for all $P_1, P_2 \in \P_n$. We say that $P_1 - P_2 = P_1' - P_2'$ if $P_1 + P_2' = P_1' + P_2$. The elements of $\V_n$ are usually called {\it virtual polytopes}. 

Let $\vol_n$ denote the ($n$-dimnesional) volume function on $\P_n$. That is, $\vol_n$ assigns to a polytope $P$ its $n$-dimensional volume $\vol_n(P)$. One knows that $\vol_n$ is a homogeneous polynomial of degree $n$ on $\P_n$. This means that for any finite collection of polytopes $P_1, \ldots, P_s \in \P$ the function $F: \R^s \to \R$ defined by:
$$F(c_1, \ldots, c_s) = \vol_n(c_1 P_1 + \cdots + c_s P_s),$$
is a homogenous polynomial of degree $n$ (possibly $0$). 
One then sees that the function $\vol_n$ extends uniquely to a homogeneous polynomial of degree $n$ on the vector space $\V_n$. We denote this extension again by $\vol_n$.

If $R \subset \R$ is a subring, we let $\P_n(R)$ denote the collection of all convex polytopes in $\R^n$ with vertices in $R^n$. We often take $R$ to be $\Q$ or $\Z$. We also let $\V_n(R)$ denote the group generated by $\P_n(R)$. 

For a convex set $P \subset \R^n$ we let $N(P)$ be the number of lattice points in $P$, that is, $N(P) = |P \cap \Z^n|$. The function $N$ restricted to the collection of lattice polytopes $\P_n(\Z)$ is a polynomial function (\cite{KhP}). In fact, for any integer $k > 0$ the function $N$ is a quasi-polynomial on $\P_n(\frac{1}{k}\Z)$. Using this one can extend $N$ to the whole $\Q$-vector space $\V_n(\Q)$. The function $N$ is usually known as the {\it Ehrhart function}.

For a polytope $P \subset \R^n$, its {\it support function} $\varphi_P$ is a function from $\R^n$ to $\R$ defined by:
$$\varphi_P(\xi) = \min\{ \langle x, \xi \rangle \mid x \in P\}.$$
One sees that $\varphi_P$ is a piecewise linear map on $\R^n$.

Recall that a {\it fan} $\Sigma$ in $\R^n$ is a finite collection of strictly convex polyhedral cones that intersect at their common faces. A fan $\Sigma$ is complete if the union of cones in $\Sigma$ is the whole $\R^n$. When the cones in $\Sigma$ are generated by rational vectors we call it a rational fan. We will only deal with rational fans. A fan $\Sigma$ is {\it simplicial} if every cone in the fan is simplicial. Finally $\Sigma$ is a {\it smooth} fan if it is rational simplicial and for each cone in $\Sigma$ the corresponding primitive vectors form part of a $\Z$-basis for $\Z^n$. Smooth fans correspond to smooth toric varieties. 

For a polytope $P \subset \R^n$ the {\it normal fan} of $P$ is a complete fan $\Sigma_P$ in $\R^n$ whose cones are in one-to-one correspondence with the faces of $P$. For a face $F$ of $P$ the corresponding cone $\sigma_F \in \Sigma_P$ is defined as the collection of all vectors $\xi \in \R^n$ such that the dot product $\langle \xi, \cdot \rangle$, regarded as a function on $P$, attains its minimum on the face $F$. Thus, the support function $\varphi_P$ restricted to each cone in the normal fan is linear. 

{Conversely, we say that a polytope $P$ is 
{\it normal} to a fan $\Sigma$ if the support function $\varphi_P$ restricted to each cone in the fan $\Sigma$ is linear. We point out that a given polytope can be normal to different fans. In fact, if $P$ is normal to a fan $\Sigma$ then this fan is a refinement of the normal fan of $P$.}

The collection of all cones that are normal to a given fan $\Sigma$ is closed under Minkowski addition and multiplication by positive scalars. We denote this set by $\P_\Sigma$. Similarly, for a subring $R \subset \R$ we denote the set of polytopes in $\P_\Sigma$ with vertices in $R^n$ by $\P_\Sigma(R)$. We also denote the subgroup of virtual polytopes generated by $\P_\Sigma$ (respective $\P_\Sigma(R)$) by $\V_\Sigma$ (respectively $\V_\Sigma(R)$).

Let $\Sigma(1) = \{\rho_1, \ldots, \rho_s\}$ be the set of rays in the fan $\Sigma$. For each ray $\rho_i$ let let $v_i$ be the unit vector along $\rho_i$. To each polytope $P \in \P_\Sigma$ one assigns the so-called {\it support numbers} $a_1, \ldots, a_s$ defined as follows:
$$a_i = \min\{ \langle x, v_i \rangle \mid x \in P\}.$$
In other words, $a_i$ is the value of the support function $\varphi_P$ on the unit vector $v_i$. Roughly speaking, the support numbers $a_i$ measure how far the corresponding facets of the polytope $P$ are from the origin.

As mentioned before, in this paper we only consider rational fans and rational (virtual) polytopes. In this setting, to define the support numbers of a polytope, it is more natural to use primitive vectors instead of unit vectors. More precisely, for each ray $\rho_i \in \Sigma(1)$ we take the vector $v_i$ to be the primitive vector along $\rho_i$, i.e. the smallest nonzero integer vector along the ray $\rho_i$. For a rational polytope $P$ we define its support numbers using this choice of the $v_i$. 

\begin{Rem}   \label{rem-supp-numbers-virtual-polytope}
One can show that the map that assigns to a polytope $P$ its support numbers $(a_1, \ldots, a_s)$ extends to a linear map from the vector space of virtual polytopes $\V_\Sigma$ to $\R^s$. We denote the virtual polytope with the support numbers $(a_1, \ldots, a_s)$ by 
$P(a_1, \ldots, a_s)$. One also can observe that when $\Sigma$ is a simplicial fan then this map is in fact surjective and hence gives an isomorphism of vector spaces $\V_\Sigma \cong \R^s$. Moreover, if $\Sigma$ is a rational simplicial fan, this restricts to an isomorphism $\V_\Sigma(\Q) \cong \Q^s$. Also if $\Sigma$ is a smooth fan then we have $\V_\Sigma(\Z) \cong \Z^s$.
This isomorphism can be interpreted as follows: let $(c_1, \ldots, c_s) \in \Z^s$ be an $s$-tuple of integers. Then there exist lattice polytopes $P$, $P'$ with support numbers $(a_1, \ldots, a_s)$, $(a'_1, \ldots, a'_s)$ respectively such that $(c_1, \ldots, c_s) = (a_1-a'_1, \ldots, a_s-a'_s)$.

Clearly $\P_\Sigma$ depends on the fan $\Sigma$. But we note that when $\Sigma$ is simplicial, the vector space $\V_\Sigma$ only depends on the set $\Sigma(1)$ of rays in $\Sigma$. That is, any polytope whose collection of facet normals is $\Sigma(1)$ belongs to the vector space $\V_\Sigma$.
\end{Rem}

\subsection{Linear families of polytopes}
By a linear family of polytopes we mean a subset of $\P_n(\Q)$ which is closed under Minkowski addition and multiplication by positive rational scalars. More precisely, we make the following definition:
\begin{Def}[Linear family of polytopes]    \label{def-linear-family-polytope}
Let $C \subset V$ be a full dimensional convex polyhedral cone in a finite dimensional $\Q$-vector space $V$. 
We call an {$\Q$-linear map} $\Delta: C \to \P_n(\Q)$ a {\it linear family of polytopes}. That is, for $\lambda_1, \lambda_2 \in C$ and $c_1, c_2 \in \Q_{\geq 0}$ we have:
$$\Delta(c_1\lambda_1 + c_2\lambda_2) = c_1\Delta(\lambda_1) + c_2\Delta(\lambda_2),$$ where the addition on the righthand side is Minkowski addition. We usually assume that general elements of the family have maximal dimension, i.e. if $\lambda$ lies in the interior of the cone $C$ then $\dim(\Delta(\lambda)) = n$.  
\end{Def}

We make the following important observation.
\begin{Prop} \label{prop-normal-fan-family}
Let $\Delta: C \to \P$ be a linear family of polytopes. Then for all $\lambda \in C^\circ$, the interior of $C$, the polytopes $\Delta(\lambda)$ have the same facet normals.
\end{Prop}
\begin{proof}
{Let $v_1, \ldots, v_s$ be extremal vectors in $C$ that generate $C$ as a cone. For each $i$, let $P_i = \Delta(v_i)$. Take $\gamma \in C^\circ$. We can write $\gamma = \sum_{i=1}^s c_i v_i$ where $c_i > 0$ for all $i$. One knows that each facet of $\Delta(\gamma)$ is a unique sum of faces of the $P_i$. But $\Delta(\gamma) = \sum_i c_iP_i$ and $P_i$ and $c_iP_i$ have the same normal fan. It follows that the facets of $\Delta(v_1 + \cdots + v_s)$ and $\Delta(\gamma)$ are parallel to each other and thus have the same normals.}
\end{proof}

The next proposition shows that a linear projection of a polyhedral cone into a vector space gives us a piecewise linear family. 
\begin{Prop}   \label{prop-family-piecewise-lin}
Suppose $\tilde{C}$ is a full dimensional polyhedral cone in a finite dimensional $\Q$-vector space $\tilde{V}$. Let $\pi: \tilde{V} \to V$ be a linear map onto another $\Q$-vector space $V$. Let $C = \pi(\tilde{C})$ be the polyhedral cone obtained by projecting $\tilde{C}$ into $V$. 
Define a family $\Delta: C \to \P_n(\Q)$ by $\Delta(\gamma) = \pi^{-1}(\gamma) \cap \tilde{C}$. Then $\Delta$ is a piecewise linear family, in the following sense: There is a fan $\Sigma$ supported on $C$ such that for any cone $\sigma \in \Sigma$, the map $\Delta$ is linear on $\sigma$.
\end{Prop}
\begin{proof}
First we prove a lemma.
\begin{Lem}[Sufficient condition for additivity]   \label{lem-family-piecewise-add}
Fix a fan $\Sigma$ with rays $\Sigma(1) = \{\rho_1, \ldots, \rho_m\}$ and the corresponding primitive vectors $\{v_1, \ldots, v_m\}$. For $a=(a_1, \ldots, a_m)$ let $\Delta(a) = \{x \in \R^n \ | \ \langle x, v_i \rangle \geq -a_i, 1\leq i \leq m \}$. Suppose we have
a convex cone $S \subset \R^m$ such that for all $a \in S$ the polytopes $\Delta(a)$ have the same normal fan $\Sigma$.
Then the map $\Delta: S \to \P_n$, $a \mapsto \Delta(a)$, is additive, that is, $\Delta(a+b) = \Delta(a) + \Delta(b)$, for all $a, b \in S$.
\end{Lem}
\begin{proof}
By assumption for each $a \in S$ the $v_i$ are rays in the normal fan of $\Delta(a)$. It follows that the support function $\varphi_{\Delta(a)}$ is determined by its values on the $v_i$. Now take $a, b \in S$. We have $\varphi_{\Delta(a)}(v_i) = a_i$, $\varphi_{\Delta(b)}(v_i)=b_i$ and $\varphi_{\Delta(a+b)}(v_i) = a_i+b_i$, for all $i$. Because $\varphi_{\Delta(a)}$, $\varphi_{\Delta(b)}$ and $\varphi_{\Delta(a+b)}$ are all linear on each cone in $\Sigma$ this then implies that $ \varphi_{\Delta(a+b)} = \varphi_{\Delta(a)} + \varphi_{\Delta(b)}$, and hence $\Delta(a+b) = \Delta(a) + \Delta(b)$ as required.
\end{proof}

For each $\gamma \in C$ let
$$\sigma_\gamma = \overline{\bigcap_\tau \pi(\tau^\circ)},$$ where the intersection is over all the faces $\tau$ of $\tilde{C}$ such that $\gamma \in \pi(\tau^\circ)$. One verifies that the collection of the cones $\{\sigma_\gamma \mid \gamma \in C\}$ is a fan with support $C$. But from the construction one sees that for all $\gamma' \in \sigma_\gamma^\circ$ the polytopes $\Delta(\gamma')$ have the same normal fan. The claim now follows from Lemma \ref{lem-family-piecewise-add}. 
\end{proof}

Fix a lattice $\Gamma \subset V$. We are interested in the number of lattice points $N(\Delta(\gamma))$ for $\gamma \in \Gamma$.  
In analogy with the anticanonical class of a compact complex manifold we make the following definition (see \eqref{equ-intro-antican-polytope}):
\begin{Def}[Anticanonical polytope in a family]   \label{def-anticanonical-polytope}
With notation as before, let $\Delta: C \to \P_n(\Q)$ be a linear family of polytopes. We say that $\Delta(\kappa)$, for $\kappa \in \Gamma$, is an {\it anticanonical (virtual) polytope} in the family $\Delta$ with respect to a lattice $\Gamma$, if for all $\gamma \in C^\circ \cap \Gamma$ such that $\gamma - \kappa \in C \cap \Gamma$ we have:
\begin{equation}  \label{equ-main-antican}
N(\Delta(\gamma - \kappa)) = N(\Delta^\circ(\gamma)).
\end{equation}
That is, the number of lattice points in $\Delta(\gamma - \kappa)$ is equal to the number of lattice points in the relative interior of $\Delta(\gamma)$.
\end{Def}

\begin{Rem}
In Definition \ref{def-anticanonical-polytope} it is important to assume that $\kappa$ is in the lattice $\Gamma$, otherwise there are infinitely many possibilities to choose a small enough rational vector $\kappa$ for which the equality \eqref{equ-main-antican} holds.
\end{Rem}

The following is immediate from the definition.
\begin{Prop} \label{prop-antican-single-lattice-point}
With notation as above, let $\Delta(\kappa)$ be an anticanonical polytope for a family $\Delta$ and assume $\kappa \in C^\circ$. Then $\Delta(\kappa)$ contains a single lattice point in its interior. 
\end{Prop}
\begin{proof}
By linearity of $\Delta$ we have $\Delta(0) = \{0\}$. By \eqref{equ-main-antican} applied to $\gamma = \kappa$ we have $N(\Delta^\circ(\kappa)) = N(\Delta(0)) = 1$.
\end{proof}

\begin{Rem}[A notion of a Fano family]  \label{rem-Fano-var}
We suggest a convex geometric analogue of the notion of a Fano variety. Namely we say that the family $\Delta: C \to \P_n(\Q)$ is {\it Fano} if firstly it has an anticanonical element $\Delta(\kappa)$ and secondly $\kappa \in C^\circ$. This agrees with the usual notion of a Fano variety for toric varieties if we take $C$ to be the $T$-linearized ample cone (see Section \ref{subsec-toric-var}). Also the Gelfand-Zetlin family is Fano which agrees with the fact that the flag variety is a Fano variety (see Section \ref{subsec-GZ}). 
\end{Rem}

The following simple example shows that not every linear family of polytopes has an anticanonical polytope. But for the families coming from smooth projective varieties (see the introduction), by the Serre duality, we always have an anticanonical polytope in the family (see \eqref{equ-intro-antican-linebundle2} in the introduction).

\begin{Ex}   \label{ex-family-without-canonical-polytope}
Let $V = \Q$, $C = \Q_{\geq 0}$ and $\Gamma = \Z$. The set $\P_1(\Q)$ is the collection of line segments in $\R$ with rational end points. It is easy to see that, for $\gamma \in \Z$, $\Delta(\gamma) = [0, 3\gamma]$ defines a linear family of polytopes. Suppose $\kappa$ is such that for all $\gamma \in \Z$ we have $|(0, 3\gamma) \cap \Z| = |[0, 3(\gamma - \kappa)] \cap \Z|$. This implies that $3\gamma - 2 = 3\gamma - 3\kappa$ and hence $\kappa = 2/3$ which is not an integer. So this family does not have an anticanonical polytope.
\end{Ex}

\section{Motivating examples from algebraic geometry}
\subsection{Toric varieties}   \label{subsec-toric-var}
Let $\Sigma$ be a complete rational fan in $\R^n$ and let  $X_\Sigma$ denote the corresponding toric variety. We assume $X_\Sigma$ is smooth although this assumption is in fact not necessary.

Let $\Sigma(1) = \{\rho_1, \ldots, \rho_s\}$ denote the rays in $\Sigma$. For each ray $\rho_i$ we denote the primitive vector along this ray by $v_i$. To each $\rho_i \in \Sigma(1)$ there corresponds an invariant prime divisors $D_i$ in $X_\Sigma$. The class group of $X_\Sigma$ is generated by the classes of the $D_i$. 
To an $s$-tuple $a=(a_1, \ldots, a_s)$ one corresponds the divisor $D(a) = \sum_i a_i D_{i}$. 
Consider the convex set $\Delta(a)$ defined by:
\begin{equation}   \label{equ-Delta-a}
\Delta(a) = \{ x \in \R^n \mid \langle v_i, x \rangle \geq -a_i,~ i=1, \ldots, s\}.
\end{equation}
From Proposition \ref{prop-family-piecewise-lin} it follows that there is a full dimensional rational polyhedral cone $C \subset \R^s$ such that $\Delta: C \to \P_n(\Q)$ is linear. More precisely, we can take $C$ to be the cone of all the $a$ such that the corresponding divisor $D(a)$ is basepoint free (see \cite[Chapters 4 and 6]{CLS}). It is well-known that $D(a)$ is basepoint free if and only if the piecewise linear function $\varphi_a$ on the fan $\Sigma$ defined by $\varphi_a(v_i) = a_i$ is a convex function (\cite[Theorem 6.1.7]{CLS}).
Thus, $\Delta: C \to \P_n$ is a linear family of convex polytopes. 

It is known that the anticanonical class $-K_{X_\Sigma}$ of $X_\Sigma$ is given by the divisor class of $\sum_{i=1}^s D_{i}$ (\cite[Theorem 8.2.3]{CLS}). In other words, it corresponds to the virtual polytope $\Delta(1, \ldots, 1)$ in the family $\Delta: C \to \P_n$. The interior of the cone $C$ of basepoint free divisors is the ample cone. Thus if $(1, \ldots, 1)$ lies in the interior of the cone $C$ then the anticanonical class is ample which means that $X_\Sigma$ is Fano.

\subsection{Flag variety of $\GL(n, \C)$ and the Gelfand-Zetlin polytopes}  \label{subsec-GZ}
The flag variety $\Fl(n)$ is the collection of all flags of linear subspaces in $\C^n$:
$$\{0\} \subsetneqq F_1 \subsetneqq \cdots \subsetneqq F_n = \C^n.$$
It can be identified with $G/B$ where $G = \GL(n, \C)$ or $\SL(n, \C)$ and $B$ is the Borel subgroup of upper triangular matrices. Each dominant weight $\lambda$ of $G$ can be represented as an increasing $n$-tuple of integers:
$$\lambda = (\lambda_1 \leq \cdots \leq \lambda_n).$$
We note that in the case of $G=\SL(n,\C)$ two sequences represent the same weight if their difference is a multiple of $(1, \ldots, 1)$.
Given a dominant weight $\lambda$, in their classic work \cite{GZ}, 
Gelfand and Zetlin construct a natural vector space basis for the irreducible representation $V_\lambda$ whose elements are parameterized with the lattice points $(x_{ij}) \in \Z^{n(n-1)/2}$ satisfying the following set of interlacing inequalities:
\begin{equation} \label{equ-GZ}
\left. \begin{matrix} \lambda_1 & \lambda_2 & \cdots & \cdots &
\cdots & \lambda_n \cr &&&&& \cr & x_{1,n-1} & x_{2,n-1} & \cdots &
\cdots & x_{n-1,n-1} \cr &&&&& \cr && x_{1,n-2} & x_{2,n-2} & \cdots
& x_{n-2,n-2} \cr &&&&& \cr &&& \cdots & \cdots & \cdots \cr &&&&&
\cr &&&& x_{1,2} & x_{2,2} \cr &&&&& \cr &&&&&x_{1,1}\cr
\end{matrix} \right.
\end{equation}
where the notation $$\left. \begin{matrix} a & b \cr &c \end{matrix}\right.$$
means $a \leq c \leq b$. The set of all points $(x_{ij})$ in $\R^{n(n-1)/2}$ satisfying \eqref{equ-GZ} is called the {\it Gelfand-Zetlin polytope} associated to $\lambda$ denoted by $\Delta_{\GZ}(\lambda)$. Thus, if $L_\lambda$ denotes the equivariant line bundle on $\Fl(n)$ associated to a dominant weight $\lambda$, we have:
$$\dim(H^0(\Fl(n), L_\lambda)) = |\Delta_\GZ(\lambda) \cap \Z^{n(n-1)/2}|.$$

One shows that 
$$\Delta: \Lambda^+ \to \P_{n(n-1)/2},~\lambda \mapsto \Delta_\GZ(\lambda),$$ is a linear family of convex polytopes. 

It is well-known that the anticanonical line bundle on the flag variety is given by the line bundle $L_{2\rho}$ (\cite[Proposition 22.8(iv)]{Brion-lec-flag}). Here $\rho$ is half the sum of positive roots (also equal to the sum of fundamental weights). As a sequence $2\rho$ is given by: 
$$2\rho = (n-1, n-3, \ldots, -(n-1)).$$

In Section \ref{subsec-antican-flag-var} we verify directly that the polytope $\Delta_\GZ(2\rho)$ satisfies the property \eqref{equ-main-antican} and hence it is in fact the anticanonical polytope in the GZ family. Thus we recover that $L_{2\rho}$ is the anticanonical class of the flag variety.

\subsection{Flag varieties, Schubert varieties and spherical varieties} \label{subsec-flag-var-Schubert-var}
Let $G$ be a connected reductive algebraic group with $B$ a Borel subgroup. Let $\Lambda^+$ denote the semigroup of dominant weights with $\Lambda^+_\R$ the positive Weyl chamber corresponding to the choice of $B$.
The variety $G/B$ is the (complete) flag variety of $G$.
To each weight $\lambda$ there corresponds a $G$-linearized line bundle $L_\lambda$. When $\lambda \in \Lambda^+$ the line bundle $L_\lambda$ is globally generated.

There are few different constructions known to assign to $L_\lambda$, $\lambda \in \Lambda^+$, a polytope $\Delta(\lambda)$ such that $N(\Delta(\lambda))$, the number of lattice points in $\Delta(\lambda)$, is equal to $\dim(H^0(G/B, L_\lambda))$. One such family of polytopes, generalizing the Gelfand-Zetlin polytopes for $\GL(n, \C)$, is the family of {\it string polytopes} (\cite{Littelmann, BZ}). String polytopes are intimately related to the so-called crystal bases of irreducible representations $V_\lambda$. For other extensions of Gelfand-Zetlin polytopes to reductive groups we refer the reader to \cite{Valya-NO-polytopes, Valya-NO-polytopes-flag}. 
Recall that a reduced decomposition $\w$, for $w_0$ the longest element in the Weyl group, is a sequence:
$$\w=(\alpha_{i_1}, \ldots, \alpha_{i_N})$$
of simple roots such that $w_0 = s_{\alpha_{i_1}} \cdots s_{\alpha_{i_N}}$. Here $N$ is the length of $w_0$ which is equal to the number of positive roots. Given a reduced decomposition $\w$ one constructs a rational convex polyhedral cone $\tilde{C}_\w \subset \Lambda^+_\R \times \R^N$, called the {\it string cone}, with the following property: let $pr_1:  \Lambda^+_\R \times \R^N \to \Lambda^+_\R$ be the projection on the first factor, then for any $\lambda \in \Lambda^+_\R$, the inverse image $\Delta_\w(\lambda) = pr_1^{-1}(\lambda)$ is a convex polytope, called a {\it string polytope}, such that the number of lattice points $N(\Delta_\w(\lambda))$ is equal to $\dim(H^0(G/B, L_\lambda))$.
From the fact that $\tilde{C}_\w$ is a convex polyhedral cone one shows the following (see \cite[Lemma 4.2]{AB}):
\begin{Lem} \label{lem-AB}
Given a reduced decomposition $\w$ there exists a fan $\Sigma_\w$ supported on the positive Weyl chamber $\Lambda^+_\R$ such that the map $\lambda \mapsto \Delta_\w(\lambda)$ is linear on each cone in the fan $\Sigma_\w$.
\end{Lem}

The above is in fact, a special case of Proposition \ref{prop-family-piecewise-lin}. Now if $C \subset \Lambda^+_\R$ is a maximal dimensional cone in the fan $\Sigma_\w$ then the map $\lambda \mapsto \Delta_\w(\lambda)$ satisfies the properties (i) and (ii) in the introduction. 

\begin{Rem}  \label{rem-string-polytope-additive}
We expect that one can always find a reduced decomposition $\w$ such that $\lambda \mapsto \Delta_\w(\lambda)$ is linear on the whole positive Weyl chmaber $\Lambda^+_\R$. 
\end{Rem}

\begin{Rem}
The above with minor modification can be applied to partial flag varieties $G/P$. Also more generally, one can define string polytopes for Schubert varieties (\cite{Caldero, Fujita}). It is well-known that Schubert varieties are Cohen-Macaulay.
\end{Rem}

\begin{Rem}[Newton-Okounkov polytopes of spherical varieties]  \label{rem-NO-polytope-spherical}
As above let $G$ be a connected reductive algebraic group. 
Recall that a normal $G$-variety is called {\it spherical} if a Borel subgroup of $G$ has a dense orbit. Spherical varieties are a generalization of toric varieties to varieties with reductive group actions. Beside toric varieties, the class of spherical varieties include (partial) flag varieties as well as group compactifications. One can associate convex polytopes to very ample $G$-linearized line bundles on projective spherical varieties (\cite{Okounkov, AB}). 
The construction of these polytopes uses string polytopes and hence depends on a choice of a reduced decomposition for the longest element $\w$. These polytopes satisfy the key property \eqref{equ-intro-dim-H^0} and are special cases of the very general construction of Newton-Okounkov bodies (\cite{Kiumars-string}). We refer to them as {\it Newton-Okounkov polytopes} for spherical varieties.   

For a given projective spherical variety and a fixed reduced decomposition, the family of Newton-Okounkov polytopes is not linear on the whole ample cone (see \cite[Example 3.2]{Kaveh-note-spherical}) but in light of Proposition \ref{prop-family-piecewise-lin} it is linear on some full dimensional rational convex cone in the ample cone. 
Thus our approach more generally applies to (smooth) spherical $G$-varieties. 

It can be shown that the assignment of polytopes to $G$-linearized line bundles extends to $G$-linearized Weil divisors on a projective spherical variety. Since spherical varieties are Cohen-Macaulay, our approach moreover can be applied to non-smooth projective spherical varieties as well.
\end{Rem}

\section{Polytope algebra}
\subsection{Algebra associated to a polynomial}
Following \cite{KhP} we define the algebra associated to a homogeneous polynomial on a vector space. We will apply this construction to the volume polynomial on the vector space spanned by a linear family of polytopes.
 
Let $V$ be a vector space over a field $\k$. Consider the algebra $\D = \D_V$ of constant coefficient differential operators on the vector space $V$. For a vector $v \in V$, let $L_v$ be the 
differentiation operator (directional derivative) on the space of polynomial functions on $V$ defined as follows. Let $h$ be a polynomial function on $V$. Then: 
\begin{equation}  \label{equ-L_v}
L_v(h)(x) = \lim_{t \to 0} \frac{h(x+tv) - h(x)}{t}.
\end{equation}
The algebra $\D$ is defined to be the commutative algebra generated by multiplication by scalars and by the derivatives $L_v$ for all $v \in V$. When $V \cong \k^n$ is finite dimensional, $\D$ can be realized as follows: Fix a basis for $V$ and let $(x_1, \ldots, x_n)$ denote the coordinate functions with respect to this basis. Each element of $\D$ is then a polynomial, with constant coefficients, in the differential operators $\partial/\partial x_1, \ldots, \partial/\partial x_n$. That is: $$\D = \{ f(\partial/\partial x_1, \ldots, \partial/\partial x_n) \mid f = \sum_{\alpha = (a_1, \ldots, a_n)} c_\alpha x_1^{a_1} \cdots x_n^{a_n} \in \k[x_1, \ldots, x_n]\}.$$
Thus, as an algebra $\D$ is naturally isomorphic to the algebra of polynomials on $V$. Let $f: V \to \k$ be a homogeneous polynomial function (of degree $n$). To $(V, f)$ we associate an algebra $A_f$ as follows. Let $I$ be the ideal of all differential operators $D \in \D$ such that $D \cdot f = 0$, i.e. those differential operators that annihilate $f$.
We call the quotient algebra $A_f = \D / I$, {\it the algebra associated to the polynomial $f$}. 

The ideal $I$ is homogeneous and hence the degree of polynomials give the algebra $A = A_f$ a natural $\Z_{\geq 0}$-grading. Let $A_i$ denote the $i$-th graded piece of $A$. One shows the following: 
\begin{itemize}
\item[(1)] $A_{0} \cong A_{n} \cong \k$ and $A_{i} = \{0\}$, for any $i>n$. 
\item[(2)] $A$ has Poincare duality, i.e. for any $0 \leq i \leq n$, the bilinear map $A_{i} \times A_{n-i} \to A_{n} \cong \k$ given by multiplication, is non-degenerate. Thus, we have $\dim_\k(A_{i}) = \dim_\k(A_{n-i})$.
\item[(3)] $A$ is generated as an algebra by its degree $1$ piece $A_1$.
\end{itemize}

Conversely, suppose $A=\bigoplus_{i=0}^n A_i$ is an algebra (over a field $\k$) with Poincare duality (in particular $A_0 = A_n \cong \k$) and also suppose $A$ is generated by $A_1$. Then $A$ is isomorphic to the algebra associated to $(V, f)$ where $V = A_1$ and $f: A_1 \to A_n \cong \k$ is the polynomial defined by $f(x) = x^n$, $\forall x \in A_1$ (see \cite[Theorem 1.1]{Kaveh-note-spherical} as well as \cite[Ex. 21.7]{Eisenbud}). 

\begin{Rem}  \label{rem-A_1-V}
We note that if $A$ is the algebra associated to $(V, f)$ then $A_1$, as a vector space, is naturally isomorphic to $V$, where the isomorphism sends $v \in V$ to the differentiation operator $L_v$ (see \eqref{equ-L_v}). 
\end{Rem}

\subsection{Polytope algebra associated to a fan}
Let $\Sigma$ be a rational fan in $\R^n$. Recall that $\P_\Sigma(\Q)$ (respectively $\V_\Sigma(\Q)$) denotes the collection of all rational convex polytopes (respectively rational virtual polytopes) that are normal to $\Sigma$ (Section \ref{subsec-prelim-polytopes}). Consider the volume polynomial $\vol_n: \V_n(\Q) \to \Q$. We denote the algebra associated to $(\V_\Sigma(\Q), {\vol_n})$ by $A_\Sigma$ (we take the field to be $\k = \Q$).

We recall that if $\Sigma$ is simplicial the vector space $\V_\Sigma(\Q)$ only depends on the set $\Sigma(1)$ of rays in $\Sigma$ (Remark \ref{rem-supp-numbers-virtual-polytope}). That is, any rational polytope whose collection of facet normals is $\Sigma(1)$ belongs to the vector space $\V_\Sigma(\Q)$. 

In \cite{KhP} it is shown that if $\Sigma$ is the fan of a smooth projective toric variety $X_\Sigma$ then $A_\Sigma$ is naturally isomorphic to the cohomology algebra $H^*(X_\Sigma, \Q)$. 

\subsection{Polytope algebra associated to a linear family of polytopes}  \label{subsec-polytope-alg}
Now let $\Delta: C \to \P_n(\Q)$ be a linear family of polytopes where $C$ is a full dimensional polyhedral cone in a $\Q$-vector space $V$. Let $\V_\Delta \subset \V_n(\Q)$ denote the $\Q$-vector subspace of virtual polytopes spanned by the image $\Delta(C)$. We denote the algebra associated to $(\V_\Delta, \vol_n)$ by $A_\Delta$. 
{The map which sends a virtual polytope $\Delta(\gamma) \in \V_\Delta$ to its corresponding linear differential operator $L_{\Delta(\gamma)}$ gives an isomorphism between $\V_\Delta$ and $(A_\Delta)_1$, the degree $1$ piece of the graded algebra $A_\Delta$ (Remark \ref{rem-A_1-V}).}


\begin{Rem}(Cohomology algebras as polytope algebras)   \label{rem-polytope-alg-cohomology}
In \cite{Kaveh-note-spherical} it is observed that the polytope algebra associated to the Gelfand-Zetlin linear family of polytopes is isomorphic to the cohomology ring of the flag variety (see also Section \ref{subsec-antican-flag-var}). In \cite{KST} the authors use this description of the cohomology ring of the flag variety to obtain interesting results connecting Schubert calculus and convex polytopes. 
\end{Rem}

\begin{Rem} \label{rem-polytope-alg-coh-map}
In general if we do not assume that $X_\Sigma$ is smooth, we have a natural map $\phi: A_\Sigma \to H^*(X_\Sigma, \Q)$.
\end{Rem}

\begin{Rem} \label{rem-polytope-alg-of-family}
By Proposition \ref{prop-normal-fan-family} we know that the there is a simplicial fan $\Sigma_\Delta$ such that all the polytopes in the family $\Delta$ belong to the vector space $\V_\Sigma(\Q)$. 
The inclusion $\V_\Delta \subset \V_{\Sigma_\Delta}(\Q)$ gives a natural embedding $(A_\Delta)_1 \hookrightarrow (A_{\Sigma_\Delta})_1$ of the degree $1$ pieces of the graded algebras $A_\Delta$ and $A_{\Sigma_\Delta}$. 
Using this embedding we can represent every element $\Delta(\gamma)$, $\gamma \in V$, of the family as a linear combination of the rays in the fan $\Sigma_\Delta$. 
\end{Rem}

\begin{Rem}   \label{rem-embedding-alg-homomorphism}
The embedding $(A_\Delta)_1 \hookrightarrow (A_{\Sigma_\Delta})_1$ induces an algebra homomorphism $A_\Delta \to A_{\Sigma_\Delta}$. In general this homomorphism is neither one-to-one nor onto (see \cite[Proposition 2.1]{KST})). 
\end{Rem}


\section{A criterion for anticanonical polytope}
As usual, let $\Delta: C \to \P_n(\Q)$ be a linear family of polytopes where $C$ is a full dimensional rational polyhedral cone in a finite dimensional $\Q$-vector space $V$. Also let $\Gamma \subset V$ be a full rank lattice. Let $\Sigma_\Delta$ be a simplicial fan such that all $\Delta(\gamma)$, $\gamma \in C^\circ$, have $\Sigma(1)$ as the set of facet normals and hence $\Delta(\gamma) \in \V_{\Sigma_\Delta}(\Q)$ (see Proposition \ref{prop-normal-fan-family} and Remark \ref{rem-supp-numbers-virtual-polytope}). 

The following is the main observation of this note.
\begin{Th}   \label{th-main}
Let $\kappa \in \Gamma$ be such that the polytope $\Delta(\kappa)$ is an anticanonical polytope for $(\Delta, \Gamma)$ (in the sense of Definition \ref{def-anticanonical-polytope}). Then in the polytope algebra $A_{\Sigma_\Delta}$, under the embedding $\V_\Delta \hookrightarrow \V_{\Sigma_\Delta}$, the polytope $\Delta(\kappa)$ is represented by a linear combination of rays in the fan $\Sigma_\Delta$ with all the coefficients equal to $1$ (see Remark \ref{rem-polytope-alg-of-family}). In particular, if the anticanonical polytope exists, it is unique.
\end{Th}

\begin{Rem}
We would like to emphasize the similarity between Theorem \ref{th-main}(2) and the formula for the anticanonical class of a toric variety.
\end{Rem}

\begin{proof}
Suppose $\kappa \in \Gamma$ is such that $\Delta(\kappa)$ is an anticanonical polytope for the family. By definition of the polytope algebra, to prove (2) we need to show that for any $\gamma \in C$ we have:
\begin{equation} \label{equ-anticanonical-sum-of-facets}
\frac{\partial}{\partial \kappa} \vol_n(\Delta(\gamma)) = \sum_i \vol_i(\Delta_i(\gamma)),
\end{equation}
where $\Delta_i(\gamma)$ denotes the $i$-th facet of the polytope $\Delta(\gamma)$ and $\vol_i$ is the $(n-1)$-dimensional volume in the affine span of $\Delta_i(\gamma)$ normalized with respect to the lattice $\Z^n$. We will use the number of lattice points and the defining equation \eqref{equ-main-antican} for an anticanonical polytope to show \eqref{equ-anticanonical-sum-of-facets}. Recall that for an $n$-dimensional polytope $P \in \R^n$ we have $\vol_n(P) = \lim_{m \to \infty} N(mP)/m^n$, where  
$N(mP)$ is the number of lattice points in the dilated polytope $mP$.
Let $\gamma \in C^\circ \cap (\kappa + C)$ and thus $\Delta(\gamma)$ has dimension $n$. We have:
\begin{align*}
\frac{\partial}{\partial \kappa} \vol_n(\Delta(\gamma)) &=
\lim_{t \to 0} \frac{\vol_n(\Delta(\gamma-t\kappa)) - \vol_n(\Delta(\gamma))}{t}, \\
&= \lim_{t \to 0} (\lim_{m \to \infty} \frac{N(\Delta(m\gamma))}{m^n} - \lim_{m \to \infty} \frac{N(\Delta(m(\gamma - t\kappa))}{m^n}) / t, \\  
&= \lim_{m \to \infty} \frac{N(\Delta(m\gamma)) - N(\Delta(m(\gamma - \kappa/m)))}{m^{n-1}},  \quad \textup{ letting } t=1/m \\
&= \lim_{m \to \infty} \frac{N(\Delta(m\gamma)) - N(\Delta(m\gamma - \kappa))}{m^{n-1}}, \\
&= \sum_i \vol_i(\Delta_i(\gamma)),
\end{align*}
which proves the claim.
\end{proof}

\section{Examples}   \label{sec-examples}
\subsection{Anticanonical class of a projective toric variety}  \label{subsec-antican-toric-var}
We follow notation from Section \ref{subsec-toric-var}.
Let $a=(a_1, \ldots, a_s) \in \Z^s$. Then a lattice point $x = (x_1, \ldots, x_n) \in \Z^n$ satisfies the strict inequalities:
$$\langle v_i, x \rangle > -a_i,~ i=1, \ldots, s,$$
if and only if it satisfies the inequalities:
$$\langle v_i, x\rangle \geq -a_i + 1,~i=1, \ldots, s.$$
It follows immediately that, in the family $\Delta: (a_1, \ldots, a_s) \mapsto \Delta(a_1, \ldots, a_s)$, the vector $\kappa = (1, \ldots, 1)$ satisfies the property \eqref{equ-main-antican}. Thus, $\sum_i D_i$ is the anticanonical class of the toric variety $X_\Sigma$ by Theorem \ref{th-main} (cf. \cite[Theorem 8.2.3]{CLS}). 

\subsection{Anticanonical class of flag variety}   \label{subsec-antican-flag-var} 
Recall that the Gelfand-Zetlin polytopes form a linear family of polytopes (Section \ref{subsec-flag-var-Schubert-var}).
First we make the following observation regarding the solutions of the Gelfand-Zetlin system of linear inequalities.
\begin{Lem}   \label{lem-GZ-strict}
$(x_{ij})$ is an integer solution of the system \eqref{equ-GZ} where we consider all the inequalities to be strict, if and only if $(x_{ij})$ is an integer solution of the non-strict system
\begin{equation} \label{equ-GZ-2}
\left. \begin{matrix} \lambda_1+(n-1) & \lambda_2+(n-3) & \cdots & \cdots &
\cdots & \lambda_n-(n-1) \cr &&&&& \cr & x_{1,n-1}+(n-2)& x_{2,n-1}
+(n-4)& \cdots & \cdots & x_{n-1,n-1}-(n-2) \cr &&&&& \cr &&
x_{2,n-2}+(n-3) & \cdots & \cdots & x_{n-2,n-2}-(n-3) \cr &&&&& \cr
&&& \cdots & \cdots & \cdots \cr &&&&& \cr &&&& x_{1,2}+1 &
x_{2,2}-1 \cr &&&&& \cr &&&&&x_{1,1}\cr
\end{matrix} \right.
\end{equation}
where the row $(x_{1,k}, x_{2,k}, \ldots, x_{k,k})$ is added with the vector
$$(k-1, k-3, \dots, -(k-3), -(k-1)).$$
Recall that $$\left. \begin{matrix} a & b \cr &c \end{matrix}\right.$$
means $a \geq c \geq b$.
\end{Lem}
In particular, the number of integer solutions of the strict system \eqref{equ-GZ} and the number of integer solutions of the non-strict system \eqref{equ-GZ-2} are the same.
\begin{proof}
We proceed by induction. For $n=2$ we have
$$\left. \begin{matrix} \lambda_1 & \lambda_2 \cr & x_{1,1} \end{matrix}\right.$$
with strict inequalities, that is $\lambda_1 < x_{1,1} < \lambda_2$.
But $x_{1,1}$ is an integral solution of this inequalities if and
only if $$ \lambda_1+1 \leq x_{1,1} \leq \lambda_2-1$$ as desired.
Now suppose the claim is true for $n$, we wish to establish it for
$n+1$. Suppose the $(x_{ij})$ are integer solutions for the system
$n+1$. Then the strict inequalities
$$
\left. \begin{matrix}
x_{1,n}& x_{2,n}& \cdots & \cdots & x_{n,n} \cr
& x_{1,n-1} & \cdots & \cdots & x_{n-1,n-1} \cr
\end{matrix} \right.
$$
hold. That is, $ x_{1,n} < x_{1,n-1} < x_{2,n} < \cdots < x_{n-1,n-1} < x_{n,n}$. This means that:
$$x_{1,n} +1 \leq x_{1,n-1} \leq x_{2,n} -1$$
$$x_{2,n} +1 \leq x_{2,n-1} \leq x_{3,n} -1$$  
$$\cdots$$
Now add $n-2i$ to every element of the $i$-th row for $i=1,\ldots,
n$ and we get
$$x_{1,n} + (n-1) \leq x_{1,n-1} + (n-2) \leq x_{2,n} + (n-3)$$ 
$$x_{2,n} + (n-3) \leq x_{2,n-1} + (n-4) \leq x_{3,n} + (n-5)$$
$$\cdots$$
Putting these inequalities together one obtains the non-strict
inequalities
\begin{displaymath}
\left. \begin{matrix} x_{1,n} + (n-1) & x_{2,n} + (n-3) & \cdots &
\cdots & x_{n,n} -(n-1) \cr & x_{1,n-1}+(n-2) & \cdots & \cdots &
x_{n-1,n-1}-(n-2) \cr
\end{matrix} \right.
\end{displaymath}
as desired. Conversely if the $(x_{ij})$ are integer solutions of
the non-strict system (\ref{equ-GZ-2}) arguing backwards we get
that they are solutions of the strict system (\ref{equ-GZ}).
This proves the lemma.
\end{proof}

\begin{Cor}
The anticanonical class of the flag variety $\Fl(n)$ is given by the Chern class of the line bundle $L_{2\rho}$ where $\rho$ is the sum of fundamental weights (also equal to half of the sum of positive roots) and hence can be represented by the sequence 
$(n-1, n-3, \ldots, -(n-1))$.
\end{Cor}

\begin{Rem}  \label{rem-KST-Schubert-var-antican-class}
It is well-known that the anticanonical class of the flag variety is twice the sum of Schubert classes of codimension $1$. In \cite{KST} the authors correspond to each Schubert variety a combination of certain faces of a Gelfand-Zetlin polytope. As mentioned in the introduction, our description of the anticanonical class of the flag variety exactly agrees with this face data of Schubert varieties in \cite{KST}.
\end{Rem}

\subsection{Anticanonical class of a group compactification}  \label{subsec-antican-group-compactification}
This example basically combines the previous two examples of toric varieties and the flag variety. 

As usual let $G$ be a connected complex reductive algebraic group. Let $\pi: G \to \GL(V)$ be a finite dimensional representation of $G$. Let us consider the 
induced map $\tilde{\pi}: G \to \mathbb{P}(\End(V))$. 
We let $X_\pi$ denote the closure of $\tilde{\pi}(G)$, i.e. $X_\pi = \overline{\tilde{\pi}(G)} \subset \mathbb{P}(\End(V))$. 
We will assume that $\tilde{\pi}$ is one-to-one or in other words the representation $\pi$ is projectively faithful.

The group $G \times G$ acts on $\End(V)$ by multiplication from left and right and the variety $X_\pi$ is a spherical $(G \times G)$-variety with an open orbit isomorphic to $G \cong (G \times G) / G_{\textup{diag}}$. We will refer to $X_\pi$ as a group compactification.

We will moreover assume $X_\pi$ is smooth although this assumption can be avoided and our approach works more generally for the non-smooth case as well (see the last remark in the introduction as well as Remark \ref{rem-NO-polytope-spherical}).

The convex hull of the Weyl orbit of the highest weights of the representation $\pi$ is called the {\it weight polytope} of $\pi$. Clearly, the weight polytope is invariant under the Weyl group action. We will denote the weight polytope by $P_\pi$ and 
its intersection with the positive Weyl chamber by $P^+_\pi$.

{Conversely, let $P \subset \Lambda_\R$ be a Weyl group invariant lattice polytope, that is, its vertices lie in the weight lattice $\Lambda$. To $P$ we can associate a represenation $\pi_P$ by setting $\pi_P = \bigoplus_{\lambda \in P \cap \Lambda^+} V_\lambda$ where as before $V_\lambda$ denotes the irreducible representation with highest weight $\lambda$.}  

The following is well-known (see \cite[Proposition 8]{Timashev}, \cite{Kapranov}).
\begin{Th}
The normalization of $X_\pi$ depends only on the normal fan of $P_\pi$ (which is a Weyl group invariant fan). 
\end{Th}

Let $\Sigma$ denote the normal fan of the polytope $P_\pi$. By the above, the normalization of $X_\pi$ only depends on the normal fan $\Sigma$. We denote this normalization by $X_\Sigma$. 

We can also consider the line bundle $\L_\pi$ on $X_\pi$ obtained by restricting the line bundle $\mathcal{O}(1)$ on the projective space $\mathbb{P}(\End(V))$. By abuse of notation we denote the pull-back of the line bundle $\L_\pi$ to the normalization $X_\Sigma$ also by $\L_\pi$.

Finally we define a polytope $\Delta(\pi)$ lying over the polytope $P^+(\pi)$ such that the number of lattice points in it is responsible for the dimension of global sections of the line bundle $\L_\pi$. This polytope is considered in \cite{KKh-Kazarnovskii} and \cite{KKh-Selecta} and is a special case of the Newton-Okounkov polytope of a spherical variety (Remark \ref{rem-NO-polytope-spherical}). For the sake of concreteness let us consider the case of $G = \GL(n, \C)$ or $\SL(n, \C)$. In this case, the polytope $\Delta(\pi)$ is a polytope fibered over $P^+(\pi)$ with the Gelfand-Zetlin polytopes as fibers:
\begin{equation}   \label{equ-Delta-pi}
\Delta(\pi) = \bigcup_{\lambda \in P^+(\pi)} (\{\lambda\} \times \Delta_{\GZ}(\lambda)) ~ \subset ~ \Lambda^+_\R \times \R^{n(n-1)/2}.
\end{equation}
For a general connected reductive group $G$ one defines $\Delta(\pi)$ using string polytopes (as in Remark \ref{rem-NO-polytope-spherical}).

\begin{Th}  \label{th-Delta-pi}
We have the following:
\begin{itemize}
\item[(1)] The dimension of space of global sections $H^0(X_\Sigma, \L_\pi)$ is equal to $N(\Delta(\pi))$, the number of lattice points in the polytope $\Delta(\pi)$. 
\item[(2)] The map $\pi \mapsto \Delta(\pi)$ is an additive map. That is, if $\pi_1, \pi_2$ are two representations then $\Delta(\pi_1 \otimes \pi_2) = \Delta(\pi_1) + \Delta(\pi_2)$.
\end{itemize}
\end{Th}
\begin{proof}
The part (1) is the main result of \cite{Okounkov}. The part (2) is an easy corollary of additivity of the Gelfand-Zetlin polytopes and the weight polytopes. This observation goes back to \cite[Section 3 and Example 3.7]{Kaveh-note-spherical}.
\end{proof}


Let us assume that the fan $\Sigma$ is smooth and let $\{\rho_1, \ldots, \rho_s\}$ denote the rays in the fan $\Sigma$. As in Section \ref{subsec-prelim-polytopes} let 
$P(a_1, \ldots, a_s) \subset \Lambda_\R$ be a Weyl invariant polytope normal to $\Sigma$ with the support numbers $a_1, \ldots, a_s$ (provided that such a polytope exists). Also let $\Delta(a_1, \ldots, a_s)$ be the polytope lying over $P^+(a_1, \ldots, a_s) = P(a_1, \ldots, a_s) \cap \Lambda^+_\R$ with the Gelfand-Zetlin polytopes as fibers, that is:
$$\Delta(a_1, \ldots, a_s) = \bigcup_{\lambda \in P^+(a_1, \ldots, a_s)} \{\lambda\} \times \Delta_{\GZ}(\lambda).$$ From the additivity of the map $\pi \mapsto \Delta(\pi)$ (Theorem \ref{th-Delta-pi}(2)) it follows that the map $(a_1, \ldots, a_s) \mapsto \Delta(a_1, \ldots, a_s)$ is additive and hence it extends to a linear function from $\Q^s$ to the vector space of virtual polytopes in $\Lambda^+_\R \times \R^{n(n-1)/2}$.
By Remark \ref{rem-supp-numbers-virtual-polytope} we know that for any $(a_1, \ldots, a_s) \in \Z^s$ we can find two lattice polytopes $P$, $P'$ normal to $\Sigma$ such that the virtual polytope $P - P'$ has support numbers $(a_1, \ldots, a_s)$. Let $\pi = \pi_P$, $\pi' = \pi_{P'}$ be the representations corresponding to the polytopes $P$, $P'$ respectively. We let $\L(a_1, \ldots, a_s)$ denote the line bundle $\L_\pi \otimes \L_{\pi'}^{-1}$. 

Now from Section \ref{subsec-antican-toric-var} and Section \ref{subsec-antican-flag-var} it follows that the virtual polytope $ \Delta(1, \ldots, 1)$ has the anticanonical property. That is, the following holds.
\begin{Prop} 
Let $\Sigma$ be a smooth Weyl invariant fan in $\Lambda_\R$. Let $(a_1, \ldots, a_s) \in \Z^s$ be such that there are representations $\pi$, $\pi'$ whose weight polytopes $P_\pi$, $P_\pi'$ have support numbers $(a_1, \ldots, a_s)$, $(a_1-1, \ldots, a_s-1)$ respectively. We then have:
$$N(\Delta(a_1-1, \ldots, a_s-1)) = N(\Delta^\circ(a_1, \ldots, a_s)).$$
\end{Prop}


\begin{Cor}
The anticanonical class of a group compactification $X_\Sigma$, is given by the Chern class of the line bundle $\L(1, \ldots, 1)$. 
\end{Cor}

\begin{Rem}  \label{rem-non-smooth-gp-compactification}
As in the last remark in the introduction, this approach works for non-smooth group compactifications as well to obtain similar formula for the anticanonical class. Although in this case the anticanonical class might not be a Cartier divisor. 
\end{Rem}





\end{document}